\documentclass[11pt,twoside]{article}
\usepackage{amsfonts}
\usepackage{amsmath}
{\markboth{\sl  }{\sl
 Weighted Dunkl transform inequalities}
\newtheorem{theorem}{Theorem}[section]
\newtheorem{remark}{Remark}[section]

\newtheorem{application}{Application}[section]

\newtheorem{proposition}{Proposition}[section]

\numberwithin {equation}{section}
\newenvironment{proof}{\textbf{Proof.}} {\hfill $\Box$}

\begin{document}
\title{\LARGE\bf The class $B_p$ for weighted generalized Fourier transform inequalities}
\date{}
\maketitle
\begin{abstract}
In the present paper, we prove for the Dunkl transform which
generalizes the Fourier transform weighted inequalities when the
weights belong to the well-known class $B_p$. As application, we
obtain for power weights Pitt's inequality.
\end{abstract}
\section{Introduction }
\par A key tool in the study of special functions with
reflection symmetries are Dunkl operators. The basic ingredient in
the theory of these operators are root systems and finite reflection
groups, acting on $\mathbb{R}^d$. The Dunkl operators are commuting
differential-difference operators $T_i, 1 \leq i \leq d$ associated
to an arbitrary finite reflection group $W$ on $\mathbb{R}^d$
(see[7]). These operators attached with a root system $R$ can be
considered as perturbations of the usual partial derivatives by
reflection parts. These reflection parts are coupled by parameters,
which are given in terms of a non negative multiplicity function
$k$. Dunkl theory was further developed by several mathematicians
(see [6, 14]) and later was applied and generalized in different
ways by many authors (see [1, 2]). The Dunkl kernel $E_k$ has been
introduced by C.F. Dunkl in [8]. For a family of weight functions
$w_k$ invariant under a reflection group $W$, we use the Dunkl
kernel and the measure $w_k(x)dx$ to define the generalized Fourier
transform $\mathcal{F}_k$, called the Dunkl transform, which enjoys
properties similar to those of the classical Fourier transform. If
the parameter $k\equiv0$ then $w_k(x)=1$, so that $\mathcal{F}_k$
becomes the classical Fourier transform and the $T_i, 1 \leq i \leq
d$ reduce to the corresponding partial derivatives
$\frac{\partial}{\partial x_i}, 1 \leq i \leq d$. Therefore Dunkl
analysis can be viewed as a generalization of
classical Fourier analysis (see next section, Remark 2.1).   \\

Let $\mu$ a nonnegative locally integrable function on
$(0,+\infty)$. We say that $\mu\in B_{p}$, $1< p<+\infty$ if there
is a constant $b_{p}>0$ such that for all $s>0$
\begin{eqnarray}
\int_{s}^{+\infty}\frac{\mu(t)}{t^{p}}dt \leq
b_{p}\frac{1}{s^{p}}\int_{0}^{s}\mu(t)dt.
\end{eqnarray}
In the particular case when $\mu$ is non-increasing, one has $\mu
\in B_{p}$.\\

The weighted Hardy inequality [16] (see also [9, 13]) states that if
$\mu$ and $\vartheta$ are locally integrable
 weight functions on $(0,+\infty)$ and $1<p\leq q<+\infty$, then there is a constant $c>0$ such that for
all non-increasing, non-negative Lebesgue measurable function $f$ on
$(0,+\infty)$, the inequality
\begin{eqnarray}\Big(\int_{0}^{+\infty}\Big(\frac{1}{t}\int_{0}^{t}f(s)ds\Big)^{q}\mu(t)dt\Big)^{\frac{1}{q}}\leq
 c\, \Big(\int_{0}^{+\infty}(f(t))^{p}\vartheta(t)dt\Big)^{\frac{1}{p}}\end{eqnarray}
  is satisfied if and only if
\begin{eqnarray}\displaystyle\sup_{s>0}\Big(\int_{0}^{s}\mu(t)dt\Big)^{\frac{1}{q}}
\Big(\int_{0}^{s}(\vartheta(t))dt\Big)^{-\frac{1}{p}}
<+\infty.\end{eqnarray} and
\begin{eqnarray}\displaystyle\sup_{s>0}\Big(\int_{s}^{+\infty}\frac{\mu(t)}{t^q}dt\Big)^{\frac{1}{q}}
\Big(\int_{0}^{s}\Big(\frac{1}{t}\int_{0}^{t}\vartheta(l)&dl&\Big)^{-p'}\vartheta(t)dt\Big)^{\frac{1}{p'}}
<+\infty.\end{eqnarray} Hardy's result still remains to be an
important one as it is closely related to the Hardy-Littlewood
maximal functions in harmonic analysis [17].\\

The aim of this paper is to prove under the $B_p$ condition (1.1)
and using the weight characterization of the Hardy operator,
weighted Dunkl transform inequalities for general nonnegative
locally integrable functions $u$, $v$ on $\mathbb{R}^d$,
\begin{eqnarray*}\Big(\int_{\mathbb{R}^{d}}|\mathcal{F}_{k}(f)(x)|^{q}u(x)
d\nu_{k}(x)\Big)^{\frac{1}{q}} \leq
c\,\Big(\int_{\mathbb{R}^{d}}|f(x)|^{p}v(x)
d\nu_{k}(x)\Big)^{\frac{1}{p}},\end{eqnarray*} where $1<p\leq 2 \leq
q<+\infty$ and $f\in L^p_{k,v}(\mathbb{R}^d)$.
$L^p_{k,v}(\mathbb{R}^d)$ denote the space $L^{p}(\mathbb{R}^d, v(x)
d\nu_k(x))$ with $\nu_k$ the weighted measure associated to the
Dunkl operators defined by
\begin{eqnarray*}d\nu_k(x):=w_k(x)dx\quad
\mbox{where}\;\;w_k(x) = \prod_{\xi\in R_+} |\langle
\xi,x\rangle|^{2k(\xi)}, \quad x \in \mathbb{R}^d.\end{eqnarray*}
$R_+$ being a positive root system and $\langle .,.\rangle$ the
standard Euclidean scalar product on $\mathbb{R}^d$ (see next
section). As application, we make a study of power weights in this
context. This all leads to Pitt\'{}s inequality:\\ for $1<p\leq
2\leq q<+\infty$, $-(2\gamma+d)<\alpha<0$,
$0<\beta<(2\gamma+d)(p-1)$ and $f\in L^p_{k,v}(\mathbb{R}^d)$, one
has
\begin{eqnarray*}\Big(\int_{\mathbb{R}^{d}}|\mathcal{F}_{k}(f)(x)|^{q}\|x\|^{\alpha}
d\nu_{k}(x)\Big)^{\frac{1}{q}} \leq
c\,\Big(\int_{\mathbb{R}^{d}}|f(x)|^{p}\|x\|^{\beta}
d\nu_{k}(x)\Big)^{\frac{1}{p}},\end{eqnarray*} with the index
constraint
$\frac{1}{2\gamma+d}(\frac{\alpha}{q}+\frac{\beta}{p})=1-\frac{1}{p}-\frac{1}{q}$
where $\displaystyle\gamma = \sum_{\xi \in R_+} k(\xi)$. This extend
to the Dunkl analysis some results obtained for the classical
Fourier analysis in [4]. \\

The contents of this paper are as follows. \\In section 2, we
collect some basic definitions and results about harmonic analysis
associated with Dunkl operators .\\
The section 3 is devoted to the proofs of the weighted Dunkl
transform inequalities when the weights belong to the class $B_p$.
As application, we obtain for power weights Pitt\'{}s inequality.\\

Along this paper we use $c$ to denote a suitable positive constant
which is not necessarily the same in each occurrence and we write
for $x \in \mathbb{R}^d,$ $\|x\| = \sqrt{\langle x,x\rangle}$.
Furthermore, we denote by

$\bullet\quad \mathcal{E}(\mathbb{R}^d)$ the space of infinitely
differentiable functions on $\mathbb{R}^d$.

$\bullet\quad \mathcal{S}(\mathbb{R}^d)$ the Schwartz space of
functions in $\mathcal{E}( \mathbb{R}^d)$ which are rapidly
decreasing as well as their derivatives.

$\bullet\quad \mathcal{D}(\mathbb{R}^d)$ the subspace of
$\mathcal{E}(\mathbb{R}^d)$ of compactly supported functions.
\section{Preliminaries}
 $ $ In this section, we recall some notations and
results in Dunkl
theory and we refer for more details to the surveys [15].\\

Let $W$ be a finite reflection group on $\mathbb{R}^{d}$, associated
with a root system $R$. For $\alpha\in R$, we denote by
$\mathbb{H}_\alpha$ the hyperplane orthogonal to $\alpha$. For a
given $\beta\in\mathbb{R}^d\backslash\bigcup_{\alpha\in R}
\mathbb{H}_\alpha$, we fix a positive subsystem $R_+=\{\alpha\in R:
\langle \alpha,\beta\rangle>0\}$. We denote by $k$ a nonnegative
multiplicity function defined on $R$ with the property that $k$ is
$W$-invariant. We associate with $k$ the index
$$\gamma = \sum_{\xi \in R_+} k(\xi) \geq 0,$$
and a weighted measure $\nu_k$ given by
\begin{eqnarray*}d\nu_k(x):=w_k(x)dx\quad
\mbox{ where }\;\;w_k(x) = \prod_{\xi\in R_+} |\langle
\xi,x\rangle|^{2k(\xi)}, \quad x \in \mathbb{R}^d,\end{eqnarray*}

Further, we introduce the Mehta-type constant $c_k$ by
$$c_k = \left(\int_{\mathbb{R}^d} e^{- \frac{\|x\|^2}{2}}
w_k (x)dx\right)^{-1}.$$

For every $1 \leq p \leq + \infty$, we denote respectively by
$L^p_k(\mathbb{R}^d)$, $L^p_{k,u}(\mathbb{R}^d)$,
$L^p_{k,v}(\mathbb{R}^d)$ the spaces $L^{p}(\mathbb{R}^d,
d\nu_k(x)),$ $L^{p}(\mathbb{R}^d, u(x)d\nu_k(x)),$
$L^{p}(\mathbb{R}^d, v(x)d\nu_k(x))$ and $L^p_k(
\mathbb{R}^d)^{rad}$ the subspace of those $f \in L^p_k(
\mathbb{R}^d)$ that are radial. We use respectively $\|\
\;\|_{p,k}$\,, $\|\ \;\|_{p,k,u}$\,, $\|\ \;\|_{p,k,v}$
as a shorthand for
 $\|\ \;\|_{L^p_k( \mathbb{R}^d)}$, $\|\ \;\|_{L^p_{k,u}( \mathbb{R}^d)}$, $\|\ \;\|_{L^p_{k,v}( \mathbb{R}^d)}.$ \\

By using the homogeneity of degree $2\gamma$ of $w_k$, it is shown
in [14] that for a radial function $f$ in $L^1_k ( \mathbb{R}^d)$,
there exists a function $F$ on $[0, + \infty)$ such that $f(x) =
F(\|x\|)$, for all $x \in \mathbb{R}^d$. The function $F$ is
integrable with respect to the measure $r^{2\gamma+d-1}dr$ on $[0, +
\infty)$ and we have
 \begin{eqnarray} \int_{\mathbb{R}^d}  f(x)\,d\nu_k(x)&=&\int^{+\infty}_0
\Big( \int_{S^{d-1}}f(ry)w_k(ry)d\sigma(y)\Big)r^{d-1}dr\nonumber\\
&=&
 \int^{+\infty}_0
\Big( \int_{S^{d-1}}w_k(ry)d\sigma(y)\Big)
F(r)r^{d-1}dr\nonumber\\&= & d_k\int^{+ \infty}_0 F(r)
r^{2\gamma+d-1}dr,
\end{eqnarray}
  where $S^{d-1}$
is the unit sphere on $\mathbb{R}^d$ with the normalized surface
measure $d\sigma$  and \begin{eqnarray}d_k=\int_{S^{d-1}}w_k
(x)d\sigma(x) = \frac{c^{-1}_k}{2^{\gamma +\frac{d}{2} -1}
\Gamma(\gamma + \frac{d}{2})}\;.  \end{eqnarray}

The Dunkl operators $T_j\,,\ \ 1\leq j\leq d\,$, on $\mathbb{R}^d$
associated with the reflection group $W$ and the multiplicity
function $k$ are the first-order differential- difference operators
given by
$$T_jf(x)=\frac{\partial f}{\partial x_j}(x)+\sum_{\alpha\in R_+}k(\alpha)
\alpha_j\,\frac{f(x)-f(\rho_\alpha(x))}{\langle\alpha,x\rangle}\,,\quad
f\in\mathcal{E}(\mathbb{R}^d)\,,\quad x\in\mathbb{R}^d\,,$$ where
$\rho_\alpha$ is the reflection on the hyperplane
$\mathbb{H}_\alpha$ and $\alpha_j=\langle\alpha,e_j\rangle,$
$(e_1,\ldots,e_d)$ being the canonical basis of $\mathbb{R}^d$.
\begin{remark}In the case $k\equiv0$, the weighted function $w_k\equiv1$ and the measure $\nu_k$ associated to the
Dunkl operators coincide with the Lebesgue measure. The $T_j$ reduce
to the corresponding partial derivatives. Therefore Dunkl analysis
can be viewed as a generalization of classical Fourier
analysis.\end{remark}

For $y \in \mathbb{C}^d$, the system
$$\left\{\begin{array}{lll}T_ju(x,y)&=&y_j\,u(x,y),\qquad1\leq j\leq d\,,\\  &&\\
u(0,y)&=&1\,.\end{array}\right.$$ admits a unique analytic solution
on $\mathbb{R}^d$, denoted by $E_k(x,y)$ and called the Dunkl
kernel. This kernel has a unique holomorphic extension to
$\mathbb{C}^d \times \mathbb{C}^d $. We have for all $\lambda\in
\mathbb{C}$ and $z, z'\in \mathbb{C}^d,\;
 E_k(z,z') = E_k(z',z)$,  $E_k(\lambda z,z') = E_k(z,\lambda z')$ and for $x, y
\in \mathbb{R}^d,\;|E_k(x,iy)| \leq 1$.\\

The Dunkl transform $\mathcal{F}_k$ is defined for $f \in
\mathcal{D}( \mathbb{R}^d)$ by
$$\mathcal{F}_k(f)(x) =c_k\int_{\mathbb{R}^d}f(y) E_k(-ix, y)d\nu_k(y),\quad
x \in \mathbb{R}^d.$$  We list some known properties of this
transform:
\begin{itemize}
\item[i)] The Dunkl transform of a function $f
\in L^1_k( \mathbb{R}^d)$ has the following basic property
\begin{eqnarray*}\| \mathcal{F}_k(f)\|_{\infty,k} \leq
 \|f\|_{ 1,k}\;. \end{eqnarray*}
\item[ii)] The Dunkl transform is an automorphism on the Schwartz space $\mathcal{S}(\mathbb{R}^d)$.
\item[iii)] When both $f$ and $\mathcal{F}_k(f)$ are in $L^1_k( \mathbb{R}^d)$,
 we have the inversion formula \begin{eqnarray*} f(x) =   \int_{\mathbb{R}^d}\mathcal{F}_k(f)(y) E_k( ix, y)d\nu_k(y),\quad
x \in \mathbb{R}^d.\end{eqnarray*}
\item[iv)] (Plancherel's theorem) The Dunkl transform on $\mathcal{S}(\mathbb{R}^d)$
 extends uniquely to an isometric automorphism on
$L^2_k(\mathbb{R}^d)$.
\end{itemize} Since the Dunkl transform $\mathcal{F}_k(f)$ is of strong-type $(1,\infty)$ and
$(2,2)$, then by interpolation, we get for $f \in
L^p_k(\mathbb{R}^d)$ with $1\leq p\leq 2$ and $p'$ such that
$\frac{1}{p}+\frac{1}{p'}=1$, the Hausdorff-Young inequality
\begin{eqnarray*}
\|\mathcal{F}_k(f)\|_{p',k}\leq c\,\|f\|_{p,k}.
\end{eqnarray*}
The Dunkl transform of a function in $L^1_k( \mathbb{R}^d)^{rad}$
 is also radial. More precisely, according to ([14], proposition
2.4), we have for  $x\in\mathbb{R}$, the following results:
\begin{eqnarray*}\int_{S^{d-1}}E_k(ix,y)w_k(y)d\sigma(y) =
d_k\, j_{\gamma + \frac{d}{2}-1}(\|x\|), \end{eqnarray*}
 and for $f$ be in $L^1_k(\mathbb{R}^d)^{rad}\;,$
 \begin{eqnarray}\mathcal{F}_k(f)(x) &=&\int^{+\infty}_0 \Big( \int_{S^{d-1}}E_k(-irx,
y)w_k(y)d\sigma(y)\Big) F(r)r^{2\gamma+d-1}dr\nonumber\\&=&
d_k\int^{+\infty}_0 j_{\gamma + \frac{d}{2}-1}(r\|x\|)
F(r)r^{2\gamma+d-1}dr,\end{eqnarray} where $F$ is the function
defined on $[ 0, + \infty)$ by $F(\|x\|) = f(x)$ and $j_{\gamma +
\frac{d}{2}-1}$  the normalized Bessel function of the first kind
and order $\gamma + \frac{d}{2}-1$ given by
\begin{eqnarray*}
j_{\gamma+\frac{d}{2}-1}(\lambda x) =\left
\{\begin{array}{ll}2^{\gamma+\frac{d}{2}-1}\Gamma(\gamma+\frac{d}{2})\frac{J_{\gamma+\frac{d}{2}-1}(\lambda
x)}{(\lambda x)^{\gamma+\frac{d}{2}-1}}&
\mbox{if}\; \lambda x\neq0,\\
1& \mbox{if}\;\lambda x=0\,,\end{array} \right.
\end{eqnarray*}
$\lambda\in\mathbb{C}$. Here $J_{\gamma+\frac{d}{2}-1}$ is the
Bessel function of first kind,
\begin{eqnarray}
J_{\gamma+\frac{d}{2}-1}(t)&=&\frac{(\frac{t}{2})^{\gamma+\frac{d}{2}-1}}{\sqrt{\pi}\Gamma(\gamma+\frac{d}{2}-\frac{1}{2})}
\int_{0}^{\pi}\cos(t\cos\theta)(\sin\theta)^{2\gamma+d-2}d\theta\nonumber
\\&=&C_{\gamma}t^{\gamma+\frac{d}{2}-1}
\int_{0}^{\frac{\pi}{2}}\cos(t\cos\theta)(\sin\theta)^{2\gamma+d-2}d\theta,
\end{eqnarray}
where $C_{\gamma}=\frac{1}{\sqrt{\pi}2^{\gamma+\frac{d}{2}-2}\Gamma(\gamma+\frac{d}{2}-\frac{1}{2})}$.\\
\section{Weighted Dunkl transform inequalities}
In this section, we denote by $p'$ and $q'$ respectively the
conjugates of $p$ and $q$ for $1<p\leq q<+\infty$. The proof
requires a useful well-known facts which we shall now state in the
following.
\begin{proposition} (see [16])
Let $1<p<+\infty$ and $v $ be a nonnegative function on
$(0,+\infty)$. The following are equivalent:
\begin{itemize}
  \item [i)] $v\in B_{p}$.
  \item [ii)] There is a positive constant $c$ such that for all $s>0$,\begin{eqnarray}
\displaystyle
\Big(\int_{0}^{s}v(t)dt\Big)^{\frac{1}{p}}\Big(\int_{0}^{s}\Big(\frac{1}{t}\int_{0}^{t}v(l)dl\Big)^{1-p'}dt\Big)^{\frac{1}{p'}}
&\leq & c\,s.
   \end{eqnarray}
  \end{itemize}
\end{proposition}
\begin{remark}$ $
\begin{itemize}
\item[1/] (see [5]) (Hardy's Lemma) Let $f$ and $g$ be non-negative Lebesgue measurable functions on
$(0, +\infty)$, and assume
$$\int_{0}^{t}f(s)ds\leq \int_{0}^{t}g(s)ds$$
for all $t\geq0$. If $\varphi$ is a non-negative and decreasing
function on $(0, +\infty)$, then
\begin{eqnarray}\int_{0}^{+\infty}f(s)\varphi(s)ds\leq \int_{0}^{+\infty}g(s)\varphi(s)ds.\end{eqnarray}
\item[2/] Let $f$ be a measurable function on $\mathbb{R}^{d}$. The
distribution function $D_f$ of $f$ is defined for all $s\geq0$ by
$$D_{f}(s)=\nu_k(\{x\in\mathbb{R}^{d}\,:\; |f(x)|>s\}).$$ The decreasing
rearrangement of $f$ is the function $f^*$ given for all $t\geq0$ by
$$f^{*}(t)=inf\{s\geq0 \,:\; D_{f}(s)\leq t\}.$$ We have the following
results: \\ i) Let $f\in L^p_{k}(\mathbb{R}^d)$, $1\leq p<+\infty$
then
\begin{eqnarray}\int_{\mathbb{R}^{d}}|f(x)|^{p}
d\nu_{k}(x)=p\int_{0}^{+\infty}s^{p-1}D_{f}(s)ds=\int_{0}^{+\infty}(f^{*}(t))^{p}dt.\end{eqnarray}\\
ii) (see [12], Theorems 4.6 and 4.7) Let $q\geq2$, then there exists
a constant $c>0$ such that, for all $f\in
L_{k}^{1}(\mathbb{R}^{d})+L_{k}^{2}(\mathbb{R}^{d})$ and $\,s\geq0$,
\begin{eqnarray}\int_{0}^{s}(\mathcal{F}_{k}(f)^{*}(t))^{q}dt\leq c
\int_{0}^{s}\Big(\int_{0}^{\frac{1}{t}}f^{*}(y)dy\Big)^{q}dt.\end{eqnarray}
iii) (see [5, 10, 11]) (Hardy-Littlewood rearrangement inequality)\\
Let $f$ and $\vartheta$ be non negative measurable functions on
$\mathbb{R}^{d}$, then
\begin{eqnarray}\int_{\mathbb{R}^{d}}f(x)\vartheta(x)d\nu_{k}(x)\leq\int_{0}^{+\infty}f^{*}(t)\vartheta^{*}(t)dt\end{eqnarray}
and
\begin{eqnarray}\int_{0}^{+\infty}f^{*}(t)
 \Big[\Big(\frac{1}{\vartheta}\Big)^{*}(t)\Big]^{-1} dt\leq\int_{\mathbb{R}^{d}}f(x)\vartheta(x)d\nu_{k}(x).\end{eqnarray}\end{itemize}
\end{remark}
Now, we begin with the proof of the following proposition which
gives a necessary condition.
\begin{proposition}
Let $u$, $v$ be non-negative $\nu_k$-locally integrable functions on
$\mathbb{R}^d$ and $1<p\leq 2\leq q<+\infty$. If there exists a
constant $c>0$ such that for all $f\in L_{k}^{p}(\mathbb{R}^d)$,
\begin{eqnarray}
\Big(\int_{0}^{+\infty}\big((\mathcal{F}_{k}(f))^{*}(t)\big)^{q}u^*(t)dt\Big)^{\frac{1}{q}}
\leq c\,
\Big(\int_{0}^{+\infty}\big(f^{*}(t)\big)^{p}\Big[\Big(\frac{1}{v}\Big)^{*}(t)\Big]^{-1}dt\Big)^{\frac{1}{p}},\nonumber\\&
&
  \end{eqnarray} then it is necessary that \begin{eqnarray} \displaystyle
\sup_{s>0}s\Big(\int_{0}^{\frac{1}{s}}u^*(t)dt\Big)^{\frac{1}{q}}\Big(\int_{0}^{s}\Big[\Big(\frac{1}{v}\Big)^{*}(t)\Big]^{-1}dt\Big)^{\frac{-1}{p}}<
+\infty.
  \end{eqnarray}
\end{proposition}
\begin{proof}
Put for any fixed $r>0$,
$$R=\displaystyle\Big(r\,\frac{\nu_{k}(B(0,1))}{1+(\nu_{k}(B(0,1)))^{2}}\Big)^{\frac{1}{2\gamma+d}},
$$ and take $f=\chi_{(0,R)}$ in (3.7), where $\chi_{(0,R)}$ is the
characteristic function of the interval $(0,R)$. For $s\geq0$ and by
(2.1) and (2.2), the distribution function of $f$ is
\begin{eqnarray*}
D_{f}(s)=\nu_k(\{x\in\mathbb{R}^{d}\,:\; \chi_{(0,R)}(\|x\|)>s\})
&=&\frac{d_k}{2\gamma+d}R^{2\gamma+d}\chi_{(0,1)}(s)
\\&=&\nu_k(B(0,1))R^{2\gamma+d}\chi_{(0,1)}(s)
\\&=& r'\chi_{(0,1)}(s),
\end{eqnarray*}
where \begin{eqnarray}r'= \nu_k(B(0,1))R^{2\gamma+d}=
r\,\frac{(\nu_{k}(B(0,1)))^{2}}{1+(\nu_{k}(B(0,1)))^{2}}.\end{eqnarray}
This yields for $t\geq0$,
\begin{eqnarray*}
f^{*}(t)&=&inf\{s\geq0 \,:\; D_{f}(s)\leq t\}
\\&=& \chi_{(0,r')}(t).
\end{eqnarray*}
Observe that $r'< r$, hence we have
\begin{eqnarray}
\Big(\int_{0}^{+\infty}\Big((\mathcal{F}_k(f))^*(t)\Big)^{q}u^*(t)dt\Big)^{\frac{1}{q}}&\leq&
c\,
\Big(\int_{0}^{r'}\Big[\Big(\frac{1}{v}\Big)^{*}(t)\Big]^{-1}dt\Big)^{\frac{1}{p}}\nonumber\\&\leq&
c\,
\Big(\int_{0}^{r}\Big[\Big(\frac{1}{v}\Big)^{*}(t)\Big]^{-1}dt\Big)^{\frac{1}{p}}.
\end{eqnarray}
According to (2.3), for $x\in\mathbb{R}^{d}$, we can assert that
\begin{eqnarray}
\mathcal{F}_k(f)(x) &=&c_{k}^{-1}
\int_{0}^{R}j_{\gamma+\frac{d}{2}-1}(\|x\|
t)\frac{t^{2\gamma+d-1}}{2^{\gamma+\frac{d}{2}-1}\Gamma(\gamma+\frac{d}{2})}dt\nonumber
\\&=&c_{k}^{-1}\|x\|^{\frac{2-2\gamma-d}{2}}\int_{0}^{R}J_{\gamma+\frac{d}{2}-1}(\|x\|
t)t^{\frac{2\gamma+d}{2}}dt.
\end{eqnarray}
Since $\cos(t\|x\|\cos\theta)\geq\cos1>\frac{1}{2}$, for $t\in(0,
R)$, $\|x\|\in(0, \frac{1}{R})$ and $\theta\in(0, \frac{\pi}{2})$,
then we obtain from (2.4), the estimate
\begin{eqnarray*}
J_{\gamma+\frac{d}{2}-1}(\|x\| t)&>&\frac{1}{2}\,C_{\gamma}\,(\|x\|
t)^{\gamma+\frac{d}{2}-1}
\int_{0}^{\frac{\pi}{2}}(\sin\theta)^{2\gamma+d-2}d\theta\\&=&\frac{1}{2}\,C_{\gamma}\,(\|x\|
t)^{\gamma+\frac{d}{2}-1}
\frac{\sqrt{\pi}\Gamma(\gamma+\frac{d}{2}-\frac{1}{2})}{2\Gamma(\gamma+\frac{d}{2})}\\&=&
\frac{(\|x\|
t)^{\frac{2\gamma+d-2}{2}}}{2^{\frac{2\gamma+d}{2}}\Gamma(\frac{2\gamma+d}{2})},
\end{eqnarray*}
which gives by (2.1), (2.2), (3.9), (3.11) and for $\|x\|\in(0,
\frac{1}{R})$
\begin{eqnarray}
\mathcal{F}_k(f)(x)&>&c_{k}^{-1}\|x\|^{\frac{2-2\gamma-d}{2}}\int_{0}^{R}
\frac{(\|x\|
t)^{\frac{2\gamma+d-2}{2}}}{2^{\frac{2\gamma+d}{2}}\Gamma(\frac{2\gamma+d}{2})}t^{\frac{2\gamma+d}{2}}dt
\nonumber\\&=&\frac{c_{k}^{-1}}{2^{\frac{2\gamma+d}{2}}\Gamma(\frac{2\gamma+d}{2})}\int_{0}^{R}t^{2\gamma+d-1}dt
\; =\;\frac{r'}{2}\;.
\end{eqnarray}
By the fact that
$$\{t\in(0,\frac{1}{r}):(\mathcal{F}_k(f))^*(t)>s\}=\{t\in(0,\frac{1}{r}):D_{\mathcal{F}_k(f)}(s)>t\},$$
 we
have from
 (3.3) \\$\displaystyle\Big(\int_{0}^{+\infty}\Big((\mathcal{F}_k(f))^*(t)\Big)^{q}u^*(t)dt\Big)^{\frac{1}{q}}$
\begin{eqnarray*}
&\geq&
\Big(\int_{0}^{\frac{1}{r}}\Big((\mathcal{F}_k(f))^*(t)\Big)^{q}u^*(t)dt\Big)^{\frac{1}{q}}\\&=&
\Big(q\int_{0}^{+\infty}s^{q-1}\Big(\int_{\{t\in(0,\frac{1}{r}),\;(\mathcal{F}_k(f))^*(t)>s\}}u^*(t)dt\Big)ds\Big)^{\frac{1}{q}}
\\&=&
\Big(q\int_{0}^{+\infty}s^{q-1}\Big(\int_{0}^{\min(D_{\mathcal{F}_k(f)}(s),\frac{1}{r})}u^*(t)dt\Big)ds\Big)^{\frac{1}{q}}.
\end{eqnarray*}
If $s<\frac{r'}{2}$, then by
 (3.12)\\ $ B(0,
\frac{1}{R})\subseteq\{x\in\mathbb{R}^{d}\,:\;
|\mathcal{F}_k(f)(x)|>\frac{r'}{2}\}
\subseteq\{x\in\mathbb{R}^{d}\,:\; |\mathcal{F}_k(f)(x)|>s\}, $\\
thus using (2.1) and (2.2), we have
\begin{eqnarray*}
D_{\mathcal{F}_k(f)}(s)&=& \int_{\{x\in\mathbb{R}^{d}\,:\;
|\mathcal{F}_k(f)(x)|>s\}}w_k(x) \,dx
\\&\geq&d_{k}\int_{0}^{\frac{1}{R}}\rho^{2\gamma+d-1}d\rho \\&=&
\frac{1}{r}\Big(1+(\nu_{k}(B(0,1)))^{2}\Big)>\frac{1}{r}\;,
\end{eqnarray*}
wich gives that
\begin{eqnarray*}
\Big(\int_{0}^{+\infty}\Big((\mathcal{F}_k(f))^*(t)\Big)^{q}u^*(t)dt\Big)^{\frac{1}{q}}&\geq&
\Big(q\int_{0}^{\frac{r'}{2}}s^{q-1}\Big(\int_{0}^{\frac{1}{r}}u^*(t)dt\Big)ds\Big)^{\frac{1}{q}}\\&=&
\Big(q\int_{0}^{\frac{r'}{2}}s^{q-1}ds\Big)^{\frac{1}{q}}\Big(\int_{0}^{\frac{1}{r}}u^*(t)dt\Big)^{\frac{1}{q}}\\&=&
\frac{r'}{2}\Big(\int_{0}^{\frac{1}{r}}u^*(t)dt\Big)^{\frac{1}{q}}.
\end{eqnarray*}According to (3.9) and (3.10), we deduce that\\\\$\displaystyle r\Big(\int_{0}^{\frac{1}{r}}u^*(t)dt\Big)^{\frac{1}{q}}
\Big(\int_{0}^{r}\Big[\Big(\frac{1}{v}\Big)^{*}(t)\Big]^{-1}dt\Big)^{-\frac{1}{p}}$
\begin{eqnarray*}
\qquad\qquad\leq\;
c\,\Big(\int_{0}^{+\infty}\Big((\mathcal{F}_k(f))^*(t)\Big)^{q}u^*(t)dt\Big)^{\frac{1}{q}}\Big(\int_{0}^{r}
\Big[\Big(\frac{1}{v}\Big)^{*}(t)\Big]^{-1}dt\Big)^{-\frac{1}{p}}
\leq c,
\end{eqnarray*}which gives (3.8). This completes the proof.
\end{proof}
\begin{theorem}
Let $u$, $v$ be non-negative $\nu_k$-locally integrable functions on
$\mathbb{R}^d$ and $1<p\leq 2\leq q<+\infty$. Assume $
\displaystyle\frac{1}{\Big(\frac{1}{v}\Big)^{*}} \in B_{p}$ and
\begin{eqnarray} \displaystyle
\sup_{s>0}s\Big(\int_{0}^{\frac{1}{s}}u^*(t)dt\Big)^{\frac{1}{q}}\Big(\int_{0}^{s}\Big[\Big(\frac{1}{v}\Big)^{*}(t)\Big]^{-1}dt\Big)^{\frac{-1}{p}}<
+\infty,
  \end{eqnarray} then
there exists a constant $c>0$ such that for all $f\in
L_{k}^{p}(\mathbb{R}^d)$, we have
\begin{eqnarray}
\Big(\int_{\mathbb{R}^{d}}|\mathcal{F}_{k}(f)(x)|^{q}u(x)
d\nu_{k}(x)\Big)^{\frac{1}{q}} \leq
c\,\Big(\int_{\mathbb{R}^{d}}|f(x)|^{p}v(x)
d\nu_{k}(x)\Big)^{\frac{1}{p}}.
\end{eqnarray}
\end{theorem}
\begin{proof} In order to establish this result, we need to show that
\begin{eqnarray}
\Big(\int_{0}^{+\infty}\big((\mathcal{F}_{k}(f))^{*}(t)\big)^{q}u^*(t)dt\Big)^{\frac{1}{q}}
\leq c\,
\Big(\int_{0}^{+\infty}\big(f^{*}(t)\big)^{p}\Big[\Big(\frac{1}{v}\Big)^{*}(t)\Big]^{-1}dt\Big)^{\frac{1}{p}}.\nonumber\\&
&
  \end{eqnarray} Take $f\in L_{k}^{p}(\mathbb{R}^d)$, then using (3.2) and
(3.4), we obtain
\begin{eqnarray*}
\Big(\int_{0}^{+\infty}\Big((\mathcal{F}_k(f))^*(t)\Big)^{q}u^*(t)dt\Big)^{\frac{1}{q}}\leq
c\, \Big(\int_{0}^{+\infty}\Big(\int_{0}^{\frac{1}{t}}f^*(s)
ds\Big)^{q}u^*(t)dt\Big)^{\frac{1}{q}}.
\end{eqnarray*}
If we make the change of variable $t=\frac{1}{s}$ on the right side,
we get
\begin{eqnarray*}
\Big(\int_{0}^{+\infty}\Big((\mathcal{F}_k(f))^*(t)\Big)^{q}u^*(t)dt\Big)^{\frac{1}{q}}\leq
c\, \Big(\int_{0}^{+\infty}\Big(\frac{1}{s}\int_{0}^{s}f^*(t)
dt\Big)^{q}\frac{u^*(\frac{1}{s})}{s^{2-q}}ds\Big)^{\frac{1}{q}},
\end{eqnarray*}
which gives from (1.2), (1.3) and (1.4), that the inequality (3.15)
is satisfied if and only if
\begin{eqnarray*}
\displaystyle\sup_{s>0}
\Big(\int_{0}^{s}\frac{u^*(\frac{1}{t})}{t^{2-q}}dt\Big)^{\frac{1}{q}}
\Big(\int_{0}^{s}\Big[\Big(\frac{1}{v}\Big)^{*}(t)\Big]^{-1}dt\Big)^{-\frac{1}{p}}
<+\infty
\end{eqnarray*}
and
\begin{eqnarray*}
\sup_{s>0}
\Big(\int_{0}^{+\infty}\frac{u^*(\frac{1}{t})}{t^{2}}dt\Big)^{\frac{1}{q}}
\Big(\int_{0}^{s}\Big(\frac{1}{t}\int_{0}^{t}\Big[\Big(\frac{1}{v}\Big)^{*}(l)\Big]^{-1}
dl\Big)^{-p'}\Big[\Big(\frac{1}{v}\Big)^{*}(t)\Big]^{-1}dt\Big)^{\frac{1}{p'}}
 <+\infty.
\end{eqnarray*}
In order to complete the proof, we must verify that (3.13) implies
these two conditions between the weights $u^*$ and $
\displaystyle\frac{1}{\Big(\frac{1}{v}\Big)^{*}}$. This follows
closely the argumentations of [4]. More precisely, since $u^*$ is
non-increasing, then\\ $u^*\in B_{q}$ and by (1.1), it yields
\begin{eqnarray*}
\int_{0}^{s}u^*(\frac{1}{t}){t^{q-2}}dt=\int_{\frac{1}{s}}^{+\infty}\frac{u^*(t)}{t^{q}}dt\leq
b_{q}s^q\int_{0}^{\frac{1}{s}}u^*(t)dt.
\end{eqnarray*}
Hence by (3.13), we get\\\\
$\displaystyle\Big(\int_{0}^{s}u^*(\frac{1}{t}){t^{q-2}}dt\Big)^{\frac{1}{q}}
\Big(\int_{0}^{s}\Big[\Big(\frac{1}{v}\Big)^{*}(t)\Big]^{-1}dt\Big)^{-\frac{1}{p}}$
\begin{eqnarray*}
\qquad\qquad\leq \;
b_{q}^{\frac{1}{q}}s\Big(\int_{0}^{\frac{1}{s}}u^*(t)dt\Big)^{\frac{1}{q}}
\Big(\int_{0}^{s}\Big[\Big(\frac{1}{v}\Big)^{*}(t)\Big]^{-1}dt\Big)^{-\frac{1}{p}}
<+\infty,
\end{eqnarray*}
and so we obtain the first condition.
\\ To show that the second condition is satisfied, observe that by means of a change of
variable, we have
\begin{eqnarray}
\Big(\int_{s}^{+\infty}\frac{u^*(\frac{1}{t})}{t^{2}}dt\Big)^{\frac{1}{q}}
= \Big(\int_{0}^{\frac{1}{s}}u^*(t)dt\Big)^{\frac{1}{q}}.
\end{eqnarray}
Now, define the function $G$ by $$\displaystyle
G(s)=\Big(\int_{0}^{s}\Big(\frac{1}{t}\int_{0}^{t}\Big[\Big(\frac{1}{v}\Big)^{*}(l)\Big]^{-1}dl
\Big)^{-p'}\Big[\Big(\frac{1}{v}\Big)^{*}(t)\Big]^{-1}dt\Big)^{\frac{1}{p'}},
$$ then by integration by parts, we get
\begin{eqnarray*}
G(s)&=&\Big[p'G(s)^{p'}+s^{p'}\Big(\int_{0}^{s}\Big[\Big(\frac{1}{v}\Big)^{*}(t)\Big]^{-1}dt\Big)^{1-p'}\\&&-
p'\int_{0}^{s}\Big(\frac{1}{t}\int_{0}^{t}\Big[\Big(\frac{1}{v}\Big)^{*}(l)\Big]^{-1}dl\Big)^{1-p'}dt\Big]^{\frac{1}{p'}},
\end{eqnarray*}
which implies
\begin{eqnarray*}
(p'-1)G(s)^{p'}  \leq
p'\int_{0}^{s}\Big(\frac{1}{t}\int_{0}^{t}\Big[\Big(\frac{1}{v}\Big)^{*}(l)\Big]^{-1}dl\Big)^{1-p'}dt,
\end{eqnarray*}
and so
\begin{eqnarray*}
G(s)\leq\Big(\frac{p'}{p'-1}\int_{0}^{s}\Big(\frac{1}{t}\int_{0}^{t}\Big[\Big(\frac{1}{v}\Big)^{*}(l)\Big]^{-1}dl\Big)^{1-p'}dt\Big)^{\frac{1}{p'}}.
\end{eqnarray*}
Since $ \displaystyle\frac{1}{\Big(\frac{1}{v}\Big)^{*}} \in B_{p}$,
we can invoke (3.1) and we obtain
\begin{eqnarray*} \Big(\int_{0}^{s}\Big(\frac{1}{t}\int_{0}^{t}\Big[\Big(\frac{1}{v}\Big)^{*}(l)\Big]^{-1}dl
\Big)^{-p'}\Big[\Big(\frac{1}{v}\Big)^{*}(t)\Big]^{-1}dt\Big)^{\frac{1}{p'}}
\leq  c\,s \Big(\int_{0}^{s}
\Big[\Big(\frac{1}{v}\Big)^{*}(t)\Big]^{-1}dt\Big)^{\frac{-1}{p}}.\end{eqnarray*}
Combining this inequality and (3.16), we deduce (3.15).\\
Note that $(|f|^{p})^{*}=(f^{*})^{p}$ and
$(|\mathcal{F}_{k}(f)|^{q})^{*}=((\mathcal{F}_{k}(f))^{*})^{q}$,
then applying (3.5) and (3.6)
 for the inequality (3.15), we obtain (3.14). This completes the proof.
\end{proof}
\begin{application} (Pitt's inequality)
Let $u(x)=\|x\|^{\alpha}$, $v(x)=\|x\|^{\beta}$,
$x\in\mathbb{R}^{d}$ with $\alpha<0$ and $\beta>0$. Using (2.1) and
(2.2), we have for $s\geq0$
\begin{eqnarray*}D_{u}(s)&=&\nu_{k}\Big(\{x\in\mathbb{R}^{d}\,:\;\|x\|^{\alpha}>s\}\Big)\\
&=&\nu_{k}\Big(B(0,s^{\frac{1}{\alpha}})\Big)=
\frac{d_{k}}{2\gamma+d}\;s^{\frac{2\gamma+d}{\alpha}},\end{eqnarray*}
which gives for $t\geq0$
\begin{eqnarray*}u^{*}(t)=inf\{s\geq 0\,:\; D_{u}(s)\leq t\}=
\Big(\frac{2\gamma+d}{d_{k}}\Big)^{\frac{\alpha}{2\gamma+d}}\;t^{\frac{\alpha}{2\gamma+d}}.\end{eqnarray*}
  On the other hand, Using (2.1)
and (2.2) again, we
have for $s\geq0$, \begin{eqnarray*}D_{\frac{1}{\vartheta}}(s)&=&\nu_{k}\Big(\{x\in\mathbb{R}^{d}\,:\;\|x\|^{-\beta}>s\}\Big)\\
&=&\nu_{k}\Big(B(0,s^{-\frac{1}{\beta}})\Big)=\frac{d_{k}}{2\gamma+d}\;s^{-\frac{2\gamma+d}{\beta}},\end{eqnarray*}
which gives for $t\geq0$,
\begin{eqnarray*}(\frac{1}{\vartheta})^{*}(t)=inf\{s\geq 0\,:\; D_{\frac{1}{\vartheta}}(s)\leq t\}
=\Big(\frac{2\gamma+d}{d_{k}}\Big)^{-\frac{\beta}{2\gamma+d}}\;t^{-\frac{\beta}{2\gamma+d}}.\end{eqnarray*}
For these weights and $1<p\leq 2\leq q<+\infty$, the hypothesis of
Theorem 3.1, gives  respectively that the integrals in the
$B_p$-inequality (1.1) for
 $\displaystyle\frac{1}{\Big(\frac{1}{v}\Big)^{*}}$ are finite and the boundedness condition
 (3.13) is valid
 if and only
 if
$$0<\beta<(2\gamma+d)(p-1)\quad\mbox{and} \quad\left\{\begin{array}{lll}-(2\gamma+d)<\alpha<0,\\
&&\\\frac{1}{2\gamma+d}(\frac{\alpha}{q}+\frac{\beta}{p})=1-\frac{1}{p}-\frac{1}{q}\,.\end{array}\right.$$
Under these conditions and index constraints, we obtain from Theorem
3.1 and for $f\in L^p_{k,v}(\mathbb{R}^d)$, Pitt's inequality
\begin{eqnarray*}\Big(\int_{\mathbb{R}^{d}}\|x\|^{\alpha}|\mathcal{F}_{k}(f)(x)|^{q}
d\nu_{k}(x)\Big)^{\frac{1}{q}} \leq
c\,\Big(\int_{\mathbb{R}^{d}}\|x\|^{\beta}|f(x)|^{p}
d\nu_{k}(x)\Big)^{\frac{1}{p}}.\end{eqnarray*}In particular for
$p=q=2$ and $0<\beta<2\gamma+d$, we get
\begin{eqnarray*}\Big(\int_{\mathbb{R}^{d}}\|x\|^{-\beta}|\mathcal{F}_{k}(f)(x)|^{2}
d\nu_{k}(x)\Big)^{\frac{1}{2}} \leq
c\,\Big(\int_{\mathbb{R}^{d}}\|x\|^{\beta}|f(x)|^{2}
d\nu_{k}(x)\Big)^{\frac{1}{2}}.\end{eqnarray*}In the classical
Fourier analysis, this inequality plays an important role for which
some uncertainty principles hold. One of them is the Beckner's
logarithmic uncertainty principle (see [3]).
\end{application}\begin{remark} The limiting case
$\beta=0$, $\alpha=(2\gamma+d)(p-2)$ and $1<p=q\leq2$ was obtained
in ([1], Section 4, Lemma 1) and gives the Hardy-Littlewood-Paley
inequality
\begin{eqnarray*}\Big(\int_{\mathbb{R}^{d}}\|x\|^{(2\gamma+d)(p-2)}|\mathcal{F}_{k}(f)(x)|^{p}
d\nu_{k}(x)\Big)^{\frac{1}{p}} \leq
c\,\Big(\int_{\mathbb{R}^{d}}|f(x)|^{p}
d\nu_{k}(x)\Big)^{\frac{1}{p}}.\end{eqnarray*}
\end{remark}

\end{document}